\documentclass[preprint]{elsarticle}
\usepackage{./stile2017}

\newcommand{\p}{p}  
\newcommand{\field}{\mbb F_\p}
\newcommand{\plane}{\field^2}

\newcommand{\puno}{\field\cup\{\infty\}}

\newcommand{\D}{D}
\newcommand{\Drich}{E}
\newcommand{\Dpoor}{D-E}
\newcommand{\rich}{W}
\newcommand{\richm}{{w_m}}
\newcommand{\sumrichm}{z_m}
\newcommand{\sumrich}{Z}

\newcommand{\DD}{D^\ast}
\newcommand{\DDrich}{E^\ast}
\newcommand{\uD}{D^\dagger}
\newcommand{\uDrich}{E^\dagger}

\newcommand{\Lm}{L^{(m)}}
\newcommand{\ellm}{\ell^{(m)}}

\newcommand{\RS}{H_{U,n}}

\newcommand{\HH}{\widetilde H}
\newcommand{\hg}{\widetilde g}

\begin{document}
	
	\begin{frontmatter}
		
		\title{On rich and poor directions determined by a subset of a finite plane}
		\author{Luca Ghidelli}\fnref{luca}
		\ead{luca.ghidelli@uottawa.ca}
		\address{Department of Mathematics and Statistics, University of Ottawa, Canada}
		\date{\today}

		\begin{abstract}
			We generalize to sets with cardinality more than $p$ a theorem of R\'edei and Sz\H{o}nyi on the number of directions determined by a subset $U$ of the finite plane $\mathbb F_p^2$. 
			A $U$-rich line is a line that meets $U$ in at least $\#U/p+1$ points, while a $U$-poor line is one that meets $U$ in at most $\#U/p-1$ points. The slopes of the $U$-rich and $U$-poor lines are called $U$-special directions. We show that either $U$ is contained in the union of $n=\lceil\#U/p\rceil$ lines, or it determines ``many'' $U$-special directions. 
			The core of our proof is a version of the polynomial method in which we study iterated partial derivatives of the R\'edei polynomial to take into account the ``multiplicity'' of the directions determined by $U$. 
		\end{abstract}
		
		\begin{keyword}
			\texttt    finite incidence geometry\sep polynomial method \sep affine Galois plane;\  special directions \sep lacunary polynomials \sep multiplicities 
			\MSC[2010] Primary:\ 05B25\sep 51A15 \sep 51E15;\ Secondary:\ 05A20\sep 11T06\sep 52C30 
		\end{keyword}
		
	\end{frontmatter}

\tableofcontents

\section{Introduction}\label{sec:intro}

In this paper we fix a prime number $p$ and we study the (affine, finite) Galois plane $\plane$ from the point of view of Incidence Geometry.  
Our aim is to provide an extension to a theorem of R\'edei and Sz\H{o}nyi and to stimulate some research on the ``polynomial method'' in situations where ``multiplicities'' are allowed. 
Our main result and the R\'edei-Sz\H{o}nyi theorem pertain to the intersections between a fixed set $U\subseteq \plane$ and all the lines $\ell\subseteq \plane$ considered ``up to parallelism'', as follows. 
For every straight line in the finite plane $\plane$ we may consider its  \emph{direction} (or ``slope'') so that parallel lines have the same direction. 
A set $U\subseteq \plane$ determines a direction $m$ if there exists a line $\ell\subseteq\plane$ with slope $m$ passing through two distinct points of $U$. 
Notice that there are $p+1$ possible directions, naturally identified with the elements of $\puno$. 
The following theorem is well-known: 
\begin{theorem}[R\'edei-Sz\H{o}nyi]\label{thm:main:rs}
Let $U\subseteq\plane$ with $\#U\leq p$. 
If $U$ is not contained in the support a line, then it determines at least $\left\lceil \frac{\#U+3}{2}\right\rceil$ directions.
\end{theorem}
\Cref{thm:main:rs} was proved in a seminal monograph \cite[Par.~36]{redei} by R\'edei, with a contribution of Megyesi, in the case $\#U=p$. It was later extended to $\#U\leq p$ by Sz\H{o}nyi, see \cite[Remark~5]{szonyi} and \cite[Sec.~3 and 5]{szonyi:around}. 
Our objective is to obtain an analogous statement for $\#U>p$ as well. 
Since for every given direction $m\in\puno$ there are only $p$ parallel lines with slope $m$, it is clear by a Pigeonhole argument that any set $U\subseteq \plane$ with $\#U>p$ determines all directions of the plane. 
However, we propose the following natural definitions: 
\begin{definition}\label{def:dir}
Let $U\subseteq\plane$ and let $\theta:=\#U/p$. 
We say that a line $\ell$ is $U$-rich if it meets $U$ in at least $\theta+1$ points, and that $\ell$ is $U$-poor if instead  $\#(\ell \cap U)\leq \theta-1$. 
We say that a direction $m\in\puno$ is \emph{$U$-special} if there is a line $\ell$ with slope $m$ that is either $U$-rich or $U$-poor. 
\end{definition}
Notice that, when $\theta\leq 1$, a $U$-special direction is nothing but a direction determined by $U$. 
Our main result reads as follows
\begin{theorem}\label{thm:main}
Let $U\subseteq\plane$ with $\# U = np-r$, for some $1\leq n\leq p$ and $0\leq r <p$. 
Then there are at least $\left\lceil \frac{p+n+2-r}{n+1}\right\rceil$ $U$-special directions, if $U$ is not contained in the union of $n$ lines. 
\end{theorem}
Notice that for $n=1$ we recover the theorem of R\'edei and Sz\H{o}nyi. 
The case $n=2$ and $r=0$ was examined by the author in \cite{ghilu:MO}. 
The directions determined by a set $U\subseteq \mbb F_q^2$ with $\# U \leq q$, where $q$ is a power of a prime, have been studied in \cite{directions:q,directions:q:2,directions:q:less}. 
We expect that a generalization of \cref{thm:main} for $U$-generic directions in the Galois plane $\mbb F_q^2$ might be obtained with similar ideas. 
See also \cite{directions:subspaces,directions:affine} for results on the directions determined by a set $U\subseteq\mbb F_p^d$ in dimension $d\geq 3$. 
 
A widely applied technique in Incidence Galois Geometry is the so-called \emph{polynomial method}, introduced by R\'edei in his pioneering work \cite{redei}. 
In order to keep the article self-contained, we will briefly outline in \cref{sec:poly,sec:lacunary} the main definitions and results that will be needed in our application of the method. 
Our \cref{prop:lacunary} on lacunary polynomials of high degree is apparently new, although its proof is completely elementary. 
For the interested reader we refer to \cite{polynomial:survey,polynomial:higher,polynomial:extremal,polynomial:lecture,polynomial:thesis} for other expositions of the polynomial method of R\'edei and 
to \cite{otherpoly:tao,otherpoly:walsh,otherpoly:dvir,otherpoly:thesis} for other uses of polynomials for incidence problems in Combinatorics and Number Theory. 

In this work we are concerned with lines that meet $U\subseteq\plane$ in multiple points, possibly more than two. 
Thus, we were naturally led to look for a way to exploit these ``multiplicities'' in the polynomial method. 
Our crucial idea is to consider the iterated derivatives of the R\'edei-Sz\H{o}nyi polynomial with respect to its second variable (see \cref{sec:poly:RS,sec:lacunary:RS}).  
Multiplicities have been considered also by the works on multiple blocking sets, such as \cite{blocking:multiple:baer,blocking:multiple,blocking:multiple:arcs,blocking:multiple:char,blocking:nova}. 
However, in these papers the set $U$ is supposed to intersect, with a given multiplicity, lines coming from \emph{all} directions (as opposed to \emph{some} directions), and so the use of $y$-derivatives is not required.

The conclusion of \cref{thm:main} is especially strong when $n$ is small compared to $p$, but in general it is unclear whether this result is optimal. 
If $n\ll 1$ is fixed, $p-r\gg p$ and $U$ is a set with cardinality $\#U=np-r$ not contained in $n$ lines, then our theorem predicts that a positive proportion of all directions is $U$-special. 
In this range the conclusion is best-possible up to multiplication by a constant. 
On the other hand, the extreme opposite case $n=p$ is trivial because every set $U\subseteq\plane$ is contained in the union of $p$ lines. 

As a final note, we remark that when $n\leq \tfrac {p+1}{2}$ and $p\neq 2$ it is possible to find a set $U_p^{(n)}\subseteq \plane$ with cardinality  $\# U_p^{(n)}=np$, and not contained in $n$ lines, that determines exactly $\tfrac {p+3}{2}$ $U$-special directions, as follows. 
We know by \cite{redei:remarks} that the set $U_p^{(1)}:=\{(k,k^{(p-1)/2}): k\in \field\}$ determines exactly $  \tfrac {p+3}{2} $ directions of $\plane$.  
This set is not contained in one line, but we have $U_p^{(1)}\subseteq \ell_1\cup \ell_{-1}$, where $\ell_{\pm 1}:= \{(k,\pm k): k\in \field\}$. 
Then, if $n\leq \tfrac {p+1}{2}$, we can form the disjoint union $U_p^{(n)}=U_p^{(1)}\sqcup_{i=1}^{n-1} \ell^{(i)}$  with $n-1$ lines parallel to $\ell_1$.   The resulting set $U_p^{(n)}$ has the properties mentioned above. 
It would be interesting to know if better constructions are possible, so we propose the folllowing: 
\begin{problem}
	Let $U\subseteq\plane$ with $\# U = np-r$, for some $1\leq n\leq p$ and $0\leq r <p$, as in \cref{thm:main}. 
	Suppose that $U$ is not contained in the union of $n$ lines. 
	Then, are there necessarily at least $\left\lceil \frac{p+3-r}{2}\right\rceil$ $U$-special directions?
\end{problem}

\section{Number of special directions in a general case}\label{sec:special}

In this section we are going to prove \cref{thm:main} under the assumption $p-r\geq n+1$. The remaining cases will be examined in \cref{sec:case}. 
Here we also assume as a blackbox the results coming from the polynomial method. 
These are recorded in \cref{JL:poly:method} below and will be proved in \cref{sec:poly} and \cref{sec:lacunary}. 

\subsection{Setup}\label{sec:special:0}

First we prepare the stage with some notation: we fix a prime number $p$ and a set $U\subseteq\plane$ with cardinality $N:=\# U = np - r$ for some $1\leq n < p$ and $0\leq r<p-n$. 
Moreover, we assume that $U$ is not contained in the union of $n$ lines. 
To make the terminology less cumbersome, we will simply call \emph{special} a direction that is $U$-special. 
Similarly, we say that a line $\ell\subseteq \plane$ is \emph{rich} (resp. poor, special) if it is $U$-rich (resp. $U$-poor, $U$-special). 
Moreover, we complete \cref{def:dir} with the following one. 
\begin{definition}\label{def:dir:rich}
	We say that a direction $m\in\puno$ is \emph{rich} (or $U$-rich) if there is some rich line with slope $m$. 
	A direction will be called \emph{generic} (or $U$-generic) if it is not special, and we call it \emph{poor} if it is special but not rich. 
\end{definition} 

We let $\D$ be the number of special directions, let  $\Drich$ be the number of rich directions and let $\rich$ the number of rich lines. 
In particular, notice that $\Drich\leq \rich$ and that there are $\D-\Drich$ poor directions. 
Finally, we assume that there are strictly less than $\left\lceil\tfrac{p+n+2-r}{n+1}\right\rceil$ special directions or, equivalently, that 
\begin{equation}\label{hyp:dir}
	D\leq 1 + \frac {p-r}{n+1}.
\end{equation}

Given the above assumptions, we want to find a contradiction. 
To do so, we introduce the following quantities. 

\begin{definition}\label{def:cm}
	For every direction $m\in\puno$ we let $c_m$ be the number of unordered pairs of points $u,v\in U$ such that the line joining them has slope $m$. 
\end{definition}

The idea will be to estimate $c_m$ separately for generic, poor and rich directions and then to compare these inequalities with the following identity: 
\begin{equation}\label{cm:total}
\sum_{m\in\puno} c_m = \binom {N} {2} 
= \frac 1 2 \left( 
n^2p^2 - npr - np +r^2 + r
\right).
\end{equation} 

%


\subsection{Preliminary considerations} \label{sec:special:1}

%

Our first observation is that $U$ is contained in the union of the rich lines.

\begin{lemma}\label{E:rich:everypoint}
	If $U\subseteq \plane$ has cardinality $N=np-r$ with $p-r\geq n+1$, then every point of $U$ is contained in a rich line. 	
\end{lemma}

\begin{proof}
	Pick $v\in U$ arbitrarily and notice that $U\setminus\{v\}$ is the disjoint union of the sets $(\ell\cap U) \setminus\{v\}$ as $\ell$ ranges through the $p+1$ lines passing through $v$. 
	If $v$ is not contained in a rich line, then $\#(\ell\cap U)\leq n$ for each such line, and so 
	\[
	N-1 \leq (p+1)(n-1) = np-p +n-1, 
	\]
	which contradicts the assumption $p-r\geq n+1$. 
\end{proof}

For every rich direction $m$ we denote by $\richm$ the number of rich lines with slope $m$, so that $\rich = \sum_{m \text{ rich}} \richm$. 
The polynomial method furnishes the following result.

\begin{lemma}\label{JL:poly:method}
	Let $U\subseteq\plane$ be as before, such that \cref{hyp:dir} holds and $U$ is not contained in the union of $n$ lines, and let $m$ be a rich direction. 
	Then every non-rich line with slope $m$ meets $U$ in at most $n-\richm$ points and each rich line meets $U$ in at least $p-r+n-1$ points.
	Moreover we have $\richm\leq n-1$ and $r\geq \richm(n-\richm)\geq n-1$.   
\end{lemma}

Since $U$ is contained in the union of the rich lines by \cref{E:rich:everypoint} and $U$ is not contained in the union of $n$ lines by assumption, we have that $\rich\geq n+1$. 
Moreover, by \cref{JL:poly:method} we have $\richm\leq n-1$ for every rich direction $m$, and so $\rich\leq (n-1)\Drich$. 
Since $\Drich\leq \D$ we get 
\begin{equation}\label{estimate:W}
	n+1\leq \rich  \leq (n-1)\D 
\end{equation}

For every rich direction $m$ we let $\ellm_1,\dots,\ellm_\richm$ be all the rich lines with slope $m$ and we let $\Lm_i:=\#(\ellm_i\cap U)$ for $i=1,\ldots,\richm$. 
Then, we define
\begin{equation*}
	\sumrichm:= \sum_{i=1}^\richm L_i
	\quad\quad\text{ and }\quad \quad
	\sumrich :=\sum_{m \text{ rich}} \sumrichm. 
\end{equation*}

Since $U$ is contained in the union of the rich lines, we can estimate the sum of the lengths of the rich lines as follows.

\begin{lemma}\label{estimate:Z}
	Let $Q:=\sum_{m \text{ rich}} \richm^2$. Then 
	\begin{equation}\label{eq:estimate:Z}
		\sumrich\leq N + \frac 1 2 \rich^2 - \frac 1 2 Q.
	\end{equation}
\end{lemma}

\begin{proof}
	Let $\ell_1,\ldots,\ell_\rich$ be all the rich lines. 
	By the Principle of Inclusion-Exclusion we have
	\[
		\# U  \geq \sum_{i=1}^\rich \#(\ell_i\cap U) - 
		\sum_{1\leq i<j\leq \rich} \# (\ell_i\cap \ell_j\cap U).
	\]
	Since two lines meet in a  point only when they have different slopes, we deduce that 
	\[
	 N \geq \sumrich - \frac 1 2\sum_{m \text{ rich}} (\rich-\richm)\richm,
	\] 
	which is equivalent to \cref{eq:estimate:Z}.
\end{proof}

\subsection{The key computations}\label{sec:special:2}

We are now ready to study the quantities $c_m$ introduced in \cref{def:cm} and to work them out to reach a contradiction. 
The following lemma is useful to estimate the sums that implicitly appear in the definition of the $c_m$'s.  

\begin{lemma}\label{cm:lemma}
	Let $L_1,\dots,L_\rich$ be integers satisfying $A\leq L_i\leq B$ for all $i=1,\ldots,\rich$ for some $A,B\geq 0$, and let $\sum_{i=1}^\rich L_i = \sumrich$. Then
	\begin{equation*}
		\sum_{i=1}^\rich \binom {L_i} 2 \leq  \frac 1 2 \left[(A+B-1) \sumrich - AB\rich\right],  
	\end{equation*}
	with equality if and only if $L_i\in\{A,B\}$ for all $i=1,\ldots,\rich$. 
\end{lemma}
\begin{proof}
	For every $i=1,\ldots,\rich$ we have $L_i(L_i-B)\leq A(L_i-B)$, with equality if and only if $L_i\in\{A,B\}$. 
	Therefore we have 
	\[
	L_i(L_i-1) \leq (A+B-1) L_i - AB
	\]
	by adding $L_i(B-1)$ on both sides. 
	The lemma now follows summing over $i=1,\ldots,\rich$ and dividing by 2. 
\end{proof}

We first consider the generic and poor directions.

\begin{proposition}\label{cm:lowerbound} 
	We have
	\begin{equation}\label{eq:cm:lowerbound}
		\sum_{m \text{ rich}} c_m \geq \frac {p-n}{2} N - \left ( 1-\frac 1 {n+1}\right)r(p-r) + \frac{n-1} 2 N \Drich.
	\end{equation}
\end{proposition}


\begin{proof}
 The lines with a generic slope $m$ meet $U$ in either $n$ or $n-1$ points, so 
 \begin{equation}\label{cm:generic}
   c_m=(p-r)\binom n 2 + r \binom {n-1} 2 = \frac {n-1} 2 (np-2r).
 \end{equation}
 The lines with a poor slope $m$ meet $U$ in at most $n$ points, so we see that 
 \begin{equation}\label{cm:poor}
   c_m\leq\frac N n \binom n 2 = \frac {n-1} 2 (np-r). 
 \end{equation}
 If we compare these estimates with \cref{cm:total} we obtain 
 \begin{align*}
 	\sum_{m \text{ rich}} c_m 
 	&\geq \binom N 2  - \frac {n-1} 2 \left[
 			(p+1-\D)(np-2r) + (\Dpoor)(np-r)	
 		\right] \\
 	&= 
 	\binom N 2  - \frac {n-1} 2 (p+1)(np-2r) -\frac {n-1} 2 \D r + 
 	\frac{n-1} 2  N \Drich.
 \end{align*}
 Using the fact that $\D-1\leq (p-r)/(n+1)$ by \cref{hyp:dir} and 
 \begin{equation*}
 	\binom N 2  - \frac {n-1} 2 (p+1)(np-2r) - \frac {n-1} 2 r = 
 	\frac {p-n}2 N + \frac {r(p-r)}2, 
 \end{equation*}
 we get the inequality \ref{eq:cm:lowerbound}.
\end{proof}

Next we consider the estimate, in the other direction, coming from the rich lines. 
\begin{proposition}\label{cm:upperbound}
	We have
	\begin{equation} \label{eq:cm:upperbound}
	\sum_{m \text{ rich}} c_m \leq \frac 1 2  {\Big (} 
		(n-1) N \Drich	+ (2p - r -1) N + f(W)
	{\Big)},
	\end{equation}
	where
	\begin{equation}\label{def:f}
		f(T):=\left(p-\frac r 2 - \frac {1} 2\right) T^2 - (n-1)\left(p-\frac r 2 - \frac {1} 2\right) T - p(p-r)T - NT.
	\end{equation}
\end{proposition}

\begin{proof}
	For every rich direction $m$ let $\ellm_1, \ldots,\ellm_{\richm}$ be all the rich lines with slope $m$ and let $\ellm_{\richm+1},\ldots,\ellm_\p$ be the non-rich lines with slope $m$. Then for every $i=1,\dots,p$ we let $\Lm_i = \#(\ellm_i\cap U)$. 
	For every $i=1,\ldots,\richm$ we have $p-r+n-1\leq \Lm_i\leq p$ by \cref{JL:poly:method}, so by \cref{cm:lemma} we have 
	\begin{equation*}
		2\sum_{i=1}^{\richm} \binom{\ \, \Lm_i}{2} \leq (2p-r+n-2) \sumrichm - p(p-r+n-1) \richm. 
	\end{equation*}
	Again by \cref{JL:poly:method} we have $\Lm_i\leq n-\richm$ for every $i=\richm+1,\dots,p$. 
	Therefore
	\begin{align*}
		2\sum_{i=\richm+1}^{p} \binom{\ \,\Lm_i}{2}
		 &\leq (n-\richm-1)(N-\sumrichm) \\
		 &\leq (n-1)N -  \richm N -  (n-1)\sumrichm + \richm^2p, 
	\end{align*}
	where we also noticed that $\richm\sumrichm\leq \richm^2 p$. 
	By the above two estimates and summing over $m$ we obtain
	\begin{equation}\label{eq:upper:end}
		\sum_{m \text{ rich}} c_m \leq \frac 1 2  {\Big (} 
		(n-1) N \Drich	- N\rich + (2p - r -1) \sumrich  - p(p-r+n-1)\rich + p Q 
		{\Big)}.
	\end{equation}
	By \cref{JL:poly:method} we have $\richm\leq n-1$ for all rich direction $m$, and so $Q\leq (n-1)\rich$. 
	Using \cref{estimate:Z} to simplify the $\sumrich$ in \cref{eq:upper:end} and then using this inequality $Q\leq (n-1)\rich$, we finally get \cref{eq:cm:upperbound}. 
\end{proof}

Next we observe that the worst-case scenario (with respect to the goal of getting a contradiction from the two above estimates) is represented by the case $W=n+1$. 

\begin{lemma}\label{f:max}
	Let $f(T)$ be as in \cref{def:f}. 
	Then $f(\rich)\leq f(n+1)$. 
\end{lemma}
\begin{proof}
	If $W=n+1$ there is nothing to prove. 
	Otherwise, we have 
	\begin{equation*}
		n+1< \rich\leq (n-1)\D \leq n-1+\frac {n-1}{n+1} (p-r) \leq n-3+p-r
	\end{equation*}
	by \cref{estimate:W,hyp:dir} and the assumption $p-r\geq n+1$. 
	Moreover the leading coefficient $p-(r+1)/2 = p/2 + (p-r-1)/2$ of $f(T)$ is positive, hence 
	\begin{align*}
		\frac{f(\rich)-f(n+1)}{\rich-n-1} 
		& = \left(p - \frac r 2 - \frac 1 2  \right) {\Big[}(\rich+n+1) - (n-1) {\Big]} - p (p-r) - N\\
		&\leq \left(p - \frac r 2 - \frac 1 2  \right) (p-r+n-1) - p (p-r) - N\\
		& = -\frac 3 2 (p-r) - \frac{n-1} 2 (r+1) - \frac {r(p-r)} 2. 
	\end{align*}
	This last expression is manifestly nonnegative, so $f(\rich) \leq f(n+1)$. 	
\end{proof}

By \cref{cm:lowerbound,cm:upperbound,f:max} we get
\begin{equation}\label{end:end}
	(p-n) N - 2r(p-r) + \frac {2 r (p-r)}{n+1}\leq (2p-r-1) N + f(n+1), 
\end{equation}
where, after some simplification: 
\[
f(n+1) = (n+1)[-p(p-r)-np +2p-1] .
\]
We notice that 
\[
p(p-r)(n+1) -(2p-r) N + p N -2r (p-r) = (p-r)^2
\]
so \cref{end:end} implies
\begin{equation}\label{end:end:end}
(p-r)^2 + \frac {2 r (p-r)}{n+1}\leq -(n+1)(n-2)p - (n+1).
\end{equation}
If $n\geq 2$ the right-hand side of \cref{end:end:end} is negative, so this inequality is impossible. 
Also for $n=1$ it is, because in this case $p-r\geq n+1=2$ and so \cref{end:end:end} implies
\[
2p\leq (p-r)^2 + 2r(p-r)\leq 2p-2.
\]

\section{The polynomial method}\label{sec:poly}

In this section introduce important tools in the polynomial method applied to ``direction problems'' of finite affine geometry.

\subsection{R\'edei polynomial and Sz\H{o}nyi complement} \label{sec:poly:redei}

The starting point of the method is the following polynomial that was introduced by R\'edei \cite{redei}. This polynomial is used to encode algebraically the multiplicity of intersection between the set $U$ and all the lines $\ell\subseteq \plane$ of the plane. 

\begin{definition}\label{def:redei}
	Given $U\subseteq\plane$ nonempty, we define the (inhomogeneous affine) \emph{R\'edei polynomial} $R_U(x,y)\in\field[x,y]$ by 
	\[
		R_U(x,y):=\prod_{(a,b)\in U} (x-ay+b).
	\]
\end{definition}
 
It has the following remarkable property: if $(a,b)\neq (a',b')$, we have
\begin{equation}\label{redei:property}
 (x-ay+b)=(x-a'y+b') \iff \frac{b-b'}{a-a'} = y.
\end{equation} 
In other words, two linear factors of $R_U$ are equal when $y$ is the slope of the line connecting $(a,b)$ and $(a',b')$. 
In particular:
\begin{remark}\label{redei:multiplicity}
For all $k,m\in\field$, the multiplicity of the linear polynomial $x-k$ within the factorization of $R_U(x,m)$ in $\field[x]$ equals the number of points of $U$ on the line $\ell_{m,k}=\{(u,v):v=mu-k\}$. 
\end{remark}

We observe that $R_U$ is a (non homogeneous) completely reducible (i.e. it factors completely as a products of linear polynomials) polynomial in two variables of total degree equal to $\#U$. 
When $\# U< \p$, Sz\H{o}nyi found an ingenious and meaningful way to complete $R_U$ to a polynomial of degree $p$. 
Our objective now is to define an analogous natural ``complement to degree $np$'' when $(n-1)p<\#U< np$, for some $n\in\N$. 

\begin{definition}\label{def:szonyi}
	Let $\mcl A$ denote the ring $\mcl A:=\field[y]$. 
	Given $U\subseteq\plane$ nonempty and $n\in\N$ with $\# U\leq np$ we define $S_{U,n}(x,y),T_{U,n}(x,y)\in\field[x,y]$ be respectively the quotient and the remainder of the univariate polynomial long division in $\mcl A[x]$ of $(x^p-x)^n$ by $R_U$:
	\[
	(x^p-x)^n = R_U(x,y) S_{U,n}(x,y) + T_{U,n}(x,y).
	\]
	We call $S_{U,n}(x,y)$ the ($n$th generalized) \emph{Sz\H{o}nyi complement} of $U$.
\end{definition}

Notice that, as a polynomial in $x$, $R_U$ is a \emph{monic} polynomial of degree $\deg_x R_U=\# U$, so the long division is well-defined and $S_{U,n}$ is again a monic polynomial in $x$, with degree $\deg_x S_{U,n} = np-\# U$.

\subsection{The RS-polynomial and its y-derivatives}\label{sec:poly:RS}

We are now able to introduce our main object of study.

\begin{definition}\label{def:RS}
	Let $U\subseteq\plane$ be nonempty and let $n\in\N$ with $np\geq \# U$. 
	We define the $n$th \emph{R\'edei-Sz\H{o}nyi polynomial} $\RS(x,y) = x^{np}+h_1(y)x^{np-1}+\dots+h_{np}(y)$ (in short \emph{RS-polynomial}) of $U$ by 
	\[
	\RS(x,y):=R_U(x,y)S_{U,n}(x,y).
	\]
\end{definition}

By inspection of the long division algorithm in \cref{def:szonyi}, it is not difficult to check that the Sz\H{o}nyi complement $S_{U,n}(x,y)$ is a polynomial in two variables with \emph{total degree} equal to $np-\# U$. Therefore we have
\begin{remark}\label{RS:degree}
	For all $U\subseteq \plane$ and $n\in\N$ with $np\geq\# U$ the RS-polynomial $\RS$ is a polynomial in two variables with total degree equal to $np$. 
	In particular its $x$-coefficients $h_j(y)$, for $1\leq j\leq np$, are polynomials in $y$ of degree at most $j$.
\end{remark}

The following is the most fundamental property of the RS-polynomial, namely its interaction with the non-vertical non-$U$-rich directions. 

\begin{proposition}\label{RS}
	Let $U\subseteq\plane$ be nonempty, let $n\in\N$ with $np\geq \# U$ and let $m\neq \infty$.  
	If $m$ is not a $U$-rich direction, then $\RS(x,m) = (x^p-x)^n$.
\end{proposition}

\begin{proof}
	By \cref{def:szonyi} we have 
	\[
	(x^p-x)^n = R_U(x,m) S_{U,n}(x,m) + T_{U,n}(x,m)
	\]
	and from \cref{redei:multiplicity} we have $R_U(x,m)\divides\, (x^p-x)^n$. 
	Hence $R_U(x,m)\divides\, T_{U,n}(x,m)$. 
	However, $R_U(x,m)$ is monic with $\deg_x R_U(x,m)=\# U$, while $\deg_x T_{U,n}<\# U$ by definition. 
	The only way this can happen is with $T_{U,n}(x,m)=0$, so $\RS(x,m) = (x^p-x)^n$ by \cref{def:RS}.
\end{proof}

The following is another important observation that we will exploit for the proof of our main theorem: when $m$ is a $U$-generic direction, the $y$-derivatives of the RS-polynomial are divisible by suitable powers of $x^p-x$. 

\begin{proposition}\label{RS:y}
Let $U\subseteq\plane$ with $(n-1)p<\#U\leq np$ for some $n\in\N$. 
Suppose that $m\neq \infty$ is not a $U$-special direction. 
Then for every $\alpha\leq n$ we have $$(x^p-x)^{n-\alpha}\divides (\partial_y^\alpha \RS)(x,m).$$ 
\end{proposition}

\begin{proof}
Recalling \cref{def:dir}, we notice that $\RS(x,m)=(x^p-x)^n$ by \cref{RS}.  
Meoreover, by \cref{redei:multiplicity}, we have that $R_U(x,m)$ is divisible by $(x^p-x)^{n-1}$ (or even $(x^p-x)^{n}$ if $\#U=np$). 
Therefore it is possible to find some $V\subset U$ such that, for 
\[
P(x,y) := \prod_{(a,b)\in V} (x- a y + b),
\]
we have $P(x,m)=(x^p-x)^{n-1}$. 
We write $\RS(x,y)=P(x,y)Q(x,y)S_{U,n}(x,y)$ for $Q = R_U/P$ and we notice that 
\begin{equation}\label{QS}
Q(x,m)S_{U,n}(x,m)=x^p-x.
\end{equation}
By Leibniz' formula, any iterated $y$-derivative $\partial_y^j P$ is a linear combination of polynomials of the form  
\[
P_W:=\prod_{(a,b)\in W} (x- a y + b)
\]
where $W\subseteq V$ satisfies $\#W = (n-1)p-j$. 
For every such $W$ we clearly have $(x^p-x)^{n-1-j}\divides P_W(x,m)$ and so $(x^p-x)^{n-1-j}\divides \partial_y^j P$ for every $\leq n-1$. 
Again by Leibniz' formula we have 
\[
\partial_y^\alpha \RS= (\partial_y ^\alpha P)\cdot Q\cdot S_{U,n} + \sum_{j=0}^{\alpha-1} \binom{\alpha}{j} (\partial_y ^j P)\cdot \partial_y^{\alpha-j}( Q\cdot S_{U,n}).
\]
If we evaluate this identity at $y=m$ and we remember \cref{QS}, the previous discussion about $\partial_y^j P$ implies that $(x^p-x)^{n-\alpha}\divides (\partial_y^\alpha \RS)(x,m)$ as required. 
\end{proof}

\section{Lacunary polynomials}\label{sec:lacunary}

The work of R\'edei includes beautiful and remarkable, albeit elementary,  results on ``lacunary'' polynomials over finite fields. 
Here we state and generalize his main observation. 
Then we will give sufficient conditions for the RS-polynomial to be lacunary, and we will deduce combinatoral consequences from this result. 

\subsection{Lacunary polynomials and reducibility}\label{sec:lacunary:reducible}

It is well-known that the polynomials $x^p-x$ and $x^p-\alpha$ (for any $\alpha\in\field$) are completely reducible: in fact, we have
\[
x^p-x = \prod_{k\in\field} (x-k)
\quad\quad\text{and}\quad\quad
x^p-\alpha = (x-\alpha)^p.
\]
These polynomials are ``lacunary'' in the sense that most of their coefficients are zero. In other words, the products above give rise to massive cancellation coefficient-wise. 
R\'edei noticed that these are the only cases in which this happens:
\begin{proposition}[R\'edei]\label{prop:lacunary:redei}
	If $x^p+g(x)\in\field[x]$ with $\deg g<\tfrac{p}{2}$ is completely reducible, then either $g(x)$ is constant or $g(x)=-x$.
\end{proposition}
This result is best-possible as the example $x^p+g(x)=x(x^{(p-1)/2}-1)^2$ shows. 
The proof (see \cref{prop:lacunary} below) is simple and it makes use of the following lemma (the case $S(x)=1$ suffices).

\begin{lemma}\label{lemma:lacunary}
	Let $P(x)\in\field[x]$ be a polynomial that factors over $\field[x]$ as $P(x)=R(x)S(x)$ with $R(x)$ completely reducible. Then
	\[
	P(x)\text{ divides } P'(x)\cdot \wb P(x) \cdot S(x), 
	\]
	where $P'$ is the derivative of $P$ and $\wb P$ is the polynomial of degree $<p$ such that $P(x)\equiv \wb P(x) \mod {x^p-x}$.
\end{lemma}

\begin{proof}
	Let $Q(x)=\op{gcd}(R,x^p-x)$ be the ``square-free part'' of $R(x)$. 
	It is clear from the Leibniz espansion that $\tfrac{R(x)}{Q(x)}$ divides $P'(x)$.  
	On the other hand $Q(x)$ divides both $P(x)$ and $X^p-x$, so $Q(x)$ divides $\wb P(x)$. 
	Therefore $P=Q\tfrac RQ S \divides  \wb P P' S$.
\end{proof}

In \cref{lemma:lacunary} we included the ``possibly non reducible'' factor $S(x)$ because, as Sz\H{o}nyi discovered \cite{szonyi}, this allows for more general applications to combinatorics.
In the following proposition we recover the R\'edei-Sz\H{o}nyi proposition \cite{redei,szonyi} for degree=$p$ almost-reducible lacunary polynomials, and we extend it naturally to degree=$np$ via an iterative procedure.

\begin{proposition}\label{prop:lacunary}
	Let $H(x) = x^{np} + g_1(x)x^{(n-1)p}+\dots +g_{n-1}x^p+g_n(x)\in\field[x]$ such that $\deg g_1 \leq A+1$ and $\deg g_k\leq B+k$ for all $1\leq k\leq n$ and some $A,B,n\in\N$. 
	Suppose also that $H(x)=R(x)S(x)$ with $R(x)$ completely reducible and that $A+B+n+\deg S<p$. Then $H(x)$ is a product of $n$ factors of the form $x^p-x$ or $x^p-\alpha$, $\alpha\in\field$.
\end{proposition}

\begin{proof}
	We proceed by induction on $n$. 
	We have that $\deg H'\leq (n-1)p + A$ and $\deg \wb H \leq n+ B$, hence 
	\[
	\deg H  = np > (n-1)p + A + B + n + \deg S \geq \deg (H'\cdot \wb H \cdot S). 
	\]
	However we also have $H\divides H'\cdot \wb H \cdot S$ by \cref{lemma:lacunary}, so we must either have $H'=0$ or $\wb H =0$. 
	If $H'=0$ then $H=G^p$ for some $G\in\field[x]$ of degree $n$. This polynomial $G$ must be completely reducible: otherwise $S$ would be divisible by the $p$-th power of some non-linear polinomial, but this is impossible because $\deg S<p$. We conclude, in case $H'=0$, that $H = \prod_{i=1}^n (x- \alpha_i)^p$ for some $\alpha_1,\ldots,\alpha_n\in\field$. 
	If $\wb H=0$ instead, then $H(x) = (x^p-x)\cdot \HH(x)$ for some $\HH(x)\in \field[x]$. 
	If $n=1$ then $\HH(x)=1$ and we are done. 
	Otherwise, we use an induction on $n$ as follows. 
	By polynomial long division we have 
	\[
	\HH(x) = x^{(n-1)p} + \hg _1(x) x^{(n-2)p} + \dots +\hg_{n-2}(x) x^p + \hg_{n-1}(x) 
	\]  
	with 
	\[
	\hg_k(x) = x^k + \sum_{j=1}^{k} g_j(x) x^{k-j}
	\]
	for all $1\leq k\leq n-1$. 
	We notice that $\deg\hg_1 \leq A+1$ and $\deg \hg_k\leq B+k$ for all $1\leq k\leq n-1$. Moreover $\HH(x) = \widetilde R(x)\cdot \widetilde S(x)$ for some $\widetilde R\divides R$ and $\widetilde S\divides S$. 
	In particular $A+B+(n-1)+\deg \widetilde S<p$ and so by induction we have that $\HH$ is a product of $n-1$ factors of the form $x^p-x$ or $x^p-\alpha$, $\alpha \in \field$.  
	Since $H(x) = (x^p-x)\cdot \HH(x)$, the proposition follows. 
\end{proof}

\subsection{The RS polynomial is lacunary}\label{sec:lacunary:RS}

We now prove that the RS-polynomial is ``lacunary'' (i.e. many of its $x$-coefficients vanish) if $U$ has few $U$-special directions. 
The induction in the following proof is a little technical, but it is executed according to the following principles
\begin{itemize}
\item
The first proposition of \cref{sec:poly:RS} implies that the specialization of the RS-polynomial at $y=m$ is ``lacunary'' when $m$ is not a $U$-rich direction.  
In other words, $h_j(m)=0$ for most $j\leq np$.
\item 
We use the second proposition of that section to show that $\partial_y^\ell h_j(m)=0$ for suitable $\ell$ and $j$, when $m$ is a $U$-generic direction. 
In other words, that these $h_j$ vanish with a certain multiplicity. 
\item
By \cref{RS:degree} we have an upper bound for the degree:  $\deg_y h_j(y)\leq j$. As a consequence, if the $U$-generic directions are numerous enough, then several $x$-coefficients $h_j$ of $\RS$ vanish identically. 
\end{itemize}

We recall that a direction $m$ is non-vertical if $m\neq \infty$. 

\begin{proposition}\label{RS:lacunary}
Let $U\subseteq \plane$ with $(n-1)p<\#U\leq np$ for some $n\in\N$, write 
\begin{equation}\label{eq:HG}
\RS(x,y) = x^{np} + G_1(x,y) x^{(n-1)p} +\dots + G_{n}(x,y)
\end{equation}
for some $G_j(x,y)$ with $\deg_x G_j<p$, and let $\DD$ (resp. $\DDrich$) be the number of the non-vertical $U$-special (resp. $U$-rich) directions. 
Then for all $1\leq j\leq n$ we have  
\begin{equation}\label{eq:RS:lacunary}
\deg_x G_j\leq \uDrich + (j-1)\uD, 
\end{equation}
where $\uD:=\max\{\DD,1\}$ and $\uDrich:=\max\{\DDrich,1\}$. 
\end{proposition}

\begin{proof}
If $m$ is not a $U$-rich direction, then by \cref{RS} we have $\RS(x,m)=(x^p-x)^n$, and so $\deg G_j(x,m)\leq j$ for all $j\leq n$. 
In particular for every $1\leq \beta \leq p-2$ and every non-$U$-rich direction $m$ we have that $h_\beta(m)=0$, because $\deg G_1(x,m)\leq 1$. 
Therefore these polynomials $h_\beta$ vanish at $p-\uDrich$ distinct elements of $\field$. 
However, we have $\deg h_\beta(y)\leq \beta$, and so $h_\beta(y)$ is the zero polynomial for $\beta=1,\ldots,p-\uDrich-1$. 
In other words 
\begin{equation}\label{eq:G1}
	\deg_x G_1(x,y)\leq \uDrich.  
\end{equation}
We are now going to prove \cref{eq:RS:lacunary} by induction on $j$, but for clarity first we show it for $j=2$. 
We consider the derivative $\partial_y \RS$ of the RS-polynomial:
\begin{equation*}
(\partial_y \RS)(x,y) = (\partial_y G_1)(x,y)\cdot x^{(n-1)p} + \dots + (\partial_y G_n)(x,y)\cdot x^0.
\end{equation*}
We know by \cref{RS:y} that $(x^p-x)^{n-1}\divides (\partial_y \RS)(x,m)$ for all non-$U$-special direction $m\neq \infty$, so that we can write $(\partial_y \RS)(x,m) = K_1(x) \cdot (x^p-x)^{n-1}$ or 
\begin{equation*}
(\partial_y \RS)(x,m) = K_1(x)\cdot x^{(n-1)p} 
+\dots + (-1)^n K_1(x) x^n, 
\end{equation*}
for some $K_1(x)$ that depends on $m$. 
If we compare the first term of the two displayed equations above, we get $K_1(x)=\partial_y G_1(x,m)$ and so $\deg K_1(x)\leq \uDrich$ because of \cref{eq:G1}. 
By comparing the other terms of the expansions above we deduce that 
$\deg \partial_y G_j(x,m) \leq \uDrich + j-1$ for all $1\leq j\leq n$. 
In particular, this estimate for $j=2$ implies that $\partial_y h_{p+\beta}(m)=0$ for all $1\leq \beta\leq p-\uDrich-2$ and all non-vertical non-$U$-special direction $m$. 
Summing up, for these values of $\beta$, the polynomial $h_{p+\beta}$ vanishes at $p-\uDrich$ elements $m\in\field$ (i.e. the non-$U$-rich directions, by the initial discussion, because $\deg G_2(x,m)\leq 2\leq \uDrich+1$) and it vanishes with double multiplicity at $p-\uD$ elements $m\in\field$ (i.e. the non-$U$-special directions, by the last discussion, because $\partial_y G_2(x,m) \leq \uDrich+1$). 
Since $\deg h_{p+\beta} \leq p+\beta$ by \cref{RS:degree}, this forces $h_{p+\beta}\equiv 0$ for all $1\leq \beta\leq p-\uDrich-\uD-1$. 
In other words, 
\[
\deg_x G_2(x,y)\leq \uDrich+\uD.
\]
We now prove by induction on $\alpha=2,\dots,n$ that
\begin{enumerate}[(i)]
	\item $\deg_x G_\alpha(x,y) \leq \uDrich + (\alpha-1) \uD$;
	\item $\deg (\partial_y^{\alpha-1} G_j)(x,m) \leq \uDrich + (\alpha-2) \uD + j-\alpha+1$ for all $\alpha\leq j\leq n$ and all non-$U$-special direction $m\neq \infty$.
\end{enumerate}
We already proved these statements for $\alpha=2$, so now we suppose they are true up to a certain $\alpha=\alpha_0$ with $2\leq \alpha_0\leq n-1$ and we aim to prove them for $\alpha =\alpha_0+1$, with the same strategy as before.
For all non-$U$-special direction we know by \cref{RS:y} that $(x^p-x)^{n-{\alpha_0}}\divides (\partial_y^{\alpha_0} \RS)(x,m)$, so
\begin{align}
(\partial_y^{\alpha_0} \RS)(x,m) 
&= (x^p-x)^{n-{\alpha_0}} \cdot (K_1(x)x^{({\alpha_0}-1)p} + \dots + K_{\alpha_0}(x))\\
&=\sum_{j=1}^n (\sum_{i=0}^{\min\{j,\alpha_0\}} (-1)^{j-i} \binom {n-\alpha_0}{j-i}K_{i}(x) x^{j-i})\cdot  x^{(n-j)p} \label{eq:k}
\end{align}
for some $K_j\in\field[x]$ with $\deg K_j\leq p-1$. 
On the other hand we have 
\begin{equation}\label{eq:y}
(\partial_y^{\alpha_0} \RS)(x,m) 
= (\partial_y^{\alpha_0} G_1)(x,m)\cdot x^{(n-1)p} + \dots + (\partial_y^{\alpha_0} G_n)(x,m)\cdot x^0
\end{equation}
and since $\deg_x G_j(x,y) \leq \uDrich+(j-1)\uD$ for $j\leq {\alpha_0}$ by statement (i), we also have 
\[
\deg (\partial_y^{\alpha_0} G_j)(x,m) \leq \uDrich+(j-1)\uD
\] 
for $j\leq {\alpha_0}$. 
If we compare the first ${\alpha_0}$ terms of \cref{eq:k,eq:y} 
we discover that $\deg K_j\leq \uDrich+(j-1)\uD$ as well. 
By comparing the remaining terms we obtain $\deg (\partial_y^{{\alpha_0}} G_j)(x,m) \leq  \uDrich+(\alpha_0-1)\uD + (j-{\alpha_0})$ for all $j\geq {\alpha_0}+1$, that is statement (ii) for $\alpha = {\alpha_0}+1$. 
Now we consider the polynomial $h_{{\alpha_0} p + \beta}(y)\in\field[y]$ for 
\[
1\leq \beta\leq p-\uDrich - {\alpha_0}\uD-1.
\]
We have that $\deg h_{\alpha_0 p + \beta}\leq \alpha_0 p + \beta$, but also that $h_{\alpha_0 p + \beta}$ has at least $p-\uDrich$ zeros in $\field$, of which at least $p-\uD$ have multiplicity at least $\alpha_0+1$. 
We have
\[
{\alpha_0} p + \beta < p-\uDrich+ {\alpha_0}(p-\uD),
\]
so these $h_{\alpha_0 p + \beta}$ are identically zero as polynomials. 
This is equivalent to statement (i) for $\alpha= \alpha_0+1$.
\end{proof}

\subsection{Outcome of the polynomial method}\label{sec:poly:output}
 
 We now combine \cref{prop:lacunary} and \cref{RS:lacunary} to get some nontrivial information on the number of intersections between the set $U$ and the lines along a $U$-rich direction. 
 
 As in \cref{sec:special:0} let $U\subseteq\plane$ be a set with cardinality $\# U = np - r$ for some $1\leq n < p$ and $0\leq r<p-n$. 
 We let $\D$ be the number of $U$-special directions, let  $\Drich$ be the number of $U$-rich directions and assume that 
 \begin{equation}\label{hyp:dir:again}
 \D\leq 1 + \frac {p-r}{n+1}
 \end{equation}
as in \cref{hyp:dir}. 
 We now further assume that $\Drich\geq 2$ and that the vertical direction $\infty$ is a $U$-rich direction. 
 Notice that $\D\geq \Drich\geq 2$ and that we have $\uDrich=\DDrich=\Drich-1$ and $\uD=\DD=\D-1$ in \cref{RS:lacunary}. 
 Let now $m\neq \infty$ be another $U$-rich direction and seek to apply \cref{prop:lacunary} with $H=\RS(x,m)$. 
 We have $H(x)=R(x)S(x)$ where $R(x)=R_U(x,m)$ is the R\'edei polynomial and $S_{U,n}(x,m)$ is the Sz\H{o}nyi complement. 
 In particular, $R(x)$ is completely reducible by its definition. 
 We write $\RS$ as in \cref{eq:HG} and we notice, using \cref{RS:lacunary} and the inequality $\Drich\leq \D$, that 
 \begin{equation}\label{end:G:AB}
 	\deg G_1(x,m)\leq A+1
 	\quad\quad\text{ and }
 	 \deg G_j(x,m)\leq B+j
 \end{equation}
 for $A=\D-2$, $B=n(\D-2)$ and all $j=1,\ldots,n$. 
 Moreover we have that $\deg S=r$ and so
 \begin{equation}\label{end:AB}
 	 A+B+n+\deg S = (n+1)(\D-2) + r <p
 \end{equation}
 by \cref{hyp:dir:again}.  
 To sum up, we have that the hypotheses of \cref{prop:lacunary} are fulfilled. 
 Therefore the polynomial $H(x)$ can be factored as 
 \begin{equation}\label{eq:H:product}
 H(x) = (x^p-x)^{n-\richm} \prod_{i=1}^{\richm} (x-\alpha_i)^p
 \end{equation}
 for some $0\leq \richm\leq n$ and some $\alpha_i\in\field$. 
 It is easy to see that 
 \begin{lemma}\label{end:alpha}
 	The elements $\alpha_i$ are pairwise distinct and the lines $\ell_i =\{(u,v): v = m u - \alpha_i\}$ for $i=1,\ldots,\richm$ are the  $U$-rich lines with slope $m$. 
 \end{lemma}
 
 \begin{proof}
 	By \cref{redei:multiplicity} we have that for all $t\in\field$, the multiplicity of $x-t$ in the factorization of $R(x)$ is the number of points of $U$ on the line $\ell_{m,t}=\{(u,v):v=mu-t\}$. 
 	Since moreover $R(x)\divides H(x)$, we have that every line $\ell_{m,t}$ other than the $\ell_i$'s meets $U$ in at most $n$ points, so it is not $U$-rich. 
 	Conversely, we notice that $S(x)$ has degree $r$, therefore the multiplicity of $x-\alpha_i$ in the factorization of $R(x)$ is at least
 	\[
 	p - r \geq n+1
 	\] 
 	because by hypothesis $r<p-n$. This means that each line $\ell_i$ for $i=1,\ldots,\richm$ is $U$-rich. 
 	In addition to this, every line has no more than $p$ points, so we have that the multiplicity of every linear factor of $R(x)$ is at most $p$. 
 	Since $\deg S(x)=r<p$, we deduce that all linear factor of $H(x)$ appear with multiplicity at most $p+r<2p$.  
 	This implies that we cannot have $\alpha_i=\alpha_j$ for $i\neq j$, otherwise by \cref{eq:H:product} the multiplicity of $x-\alpha_i$ in the factorization of $H(x)$ would be at least $2p$.
 \end{proof}

In particular, $\richm$ is the number of $U$-rich lines with slope $m$, as in \cref{sec:special:1}. Since $m$ is a $U$-rich direction, we have that $\richm\geq 1$. 
We already remarked that $\richm\leq n$, moreover we can have $\richm=n$ only if the whole of the set $U$ is contained in the union of the $\richm=n$ parallel lines $\ell_1,\ldots,\ell_{\richm}$ described in \cref{end:alpha}. 
Finally, we observe that 
\begin{equation}\label{end:r}
	r\geq \richm(n-\richm).
\end{equation}
Indeed, every factor $x-\alpha_i$ appears with multiplicity $p+n-\richm$ in the factorization of $H(x)$. Since every linear factor of $R(x)$ has multiplicity at most $p$, we have that $(x-\alpha_i)^{n-\richm}\divides S(x)$ for all $i=1,\ldots, \richm$. 
Since $\deg S = r$, we get \cref{end:r}. 
We are now able to prove the result we used as a blackbox in \cref{sec:special}.

\begin{proof}[Proof of \protect\cref{JL:poly:method}]
	Let $U\subseteq\plane$ and let $m$ be any $U$-rich direction. 
	If $\Drich\geq 2$ then there is some other $U$-rich direction $m'$ and by a linear change of coordinates we may assume that $m'=\infty$, and so also $m\neq \infty$. Then the discussion above implies the lemma. 
	If $\Drich=1$ but $\D\geq 2$, we choose some $U$-special direction $m'\neq m$ and by a linear change of coordinates we assume that $m'=\infty$. Then we repeat the above discussion, the main difference being that now $\uDrich=\DDrich=\Drich$ in \cref{RS:lacunary}. Moreover we have $\Drich\leq \D-1$, hence we still have \cref{end:G:AB} with $A=\D-2$ and $B=n(\D-2)$ and so also the inequality \cref{end:AB}. The rest of the discussion is the same as in the case before. 
	Finally, if $\Drich=\D=1$ we simply assume by a linear change of coordinates that $m\neq \infty$. 
	Again, we repeat the reasoning of this subection. 
	Now we have, by \cref{RS:lacunary}, that \cref{end:G:AB} holds with $A=B=0$ and so \cref{end:AB} is replaced by the following computation 
	\[
	A+B+n\deg S = n+r < p, 
	\]
	which holds by the assumption $r<p-n$.  
	Therefore, the hypotheses of \cref{prop:lacunary} are still fulfilled and the lemma follows also in this case. 
\end{proof}

 \begin{remark}\label{rmk:megyesi}
 	The trick of reducing to the case where $\infty$ is $U$-special is the contribution by Megyesi mentioned in the introduction. 
 \end{remark} 

\section{Completing the proof of the main theorem}\label{sec:case}

In this section we will prove \cref{thm:main} when $p-r\leq n$. Since in this case
\[
 \left\lceil \frac{p-r+n+2}{n+1}\right\rceil = 2, 
\]
we need to show that $U$ is either contained in the union of $n$ lines or that there are at least two $U$-special directions. 
 We can assume that $n<p$ because otherwise $U$ is trivially contained in the union of $n=p$ lines. 
 In particular the assumptions $n<p$ and $p-r\leq n$ imply that $r\neq 0$. 
 Now, we suppose that there is at most one $U$-special direction $m_0\in\puno$. 
 If there is no $U$-special direction, we choose $m_0$ arbitrarily. 
 We make the following observation.
 
 \begin{lemma}\label{last:lemma}
  Every line with slope $m_0$ is either contained in $U$ or it meets $U$ in at most $p-r$ points.
 \end{lemma}

 \begin{proof}
 Let $\ell_{m_0}\subseteq\plane$ be a line with slope $m_0$, not completely contained in $U$, and let $v\in\ell\setminus U$.
 By the choice of $m_0$, we notice that every line passing through $v$, other than $\ell_{m_0}$, is not $U$-special. 
 Therefore, if $\ell_m$ denotes the line through $v$ with slope $m$, we have
 \[
  N = \sum_{m\in\puno} \#(\ell_m\cap U) \geq \#(\ell_{m_0}\cap U) + p(n-1).
 \]
 This shows that $\#(\ell_{m_0}\cap U)\leq p-r$.
 \end{proof}
 We recall the definition of $c_m$ from \cref{def:cm} and the fact that $2c_m=(n-1)(np-2r)$ for all $m\in\puno\setminus\{m_0\}$ by \cref{cm:generic}. 
 Therefore we have
 \begin{equation}\label{eq:p-r=n:1}
  c_{m_0}  = \binom {N} 2 - p\cdot \frac {n-1} 2 \cdot(np-2r) 
 \end{equation}
 by \cref{cm:total}. 
 On the other hand by \cref{last:lemma} we have 
 \begin{equation}\label{eq:p-r=n:2}
  c_{m_0} \leq \rich \binom p 2  +  \frac {p-r-1} 2(N-\rich p), 
 \end{equation}
 where $\rich$ is the number of the lines (necessarily with slope $m_0$) that are contained in $U$. 
 Moreover, according to \cref{cm:lemma}, we have the equality in \cref{eq:p-r=n:2} if and only if we have $\#(\ell\cap U) \in\{0,p-r\}$ for each line $\ell$ with slope $m_0$ not contained in $U$. 
 If we compare \cref{eq:p-r=n:1} and \cref{eq:p-r=n:2}, after due simplification, we get
 \begin{equation}\label{last:estimate}
  (n-1) pr \leq \rich p r.
 \end{equation}
 This forces $\rich=n-1$, because $\rich p\leq \# U< np$ and $r\neq 0$. 
 Consequently, we have equality in \cref{last:estimate}, and so also in  \cref{eq:p-r=n:2}.  
 This is only possible if $U$ consists of the union of $n-1$ parallel lines along $m_0$, plus $p-r$ additional points on another line with same slope. 
 In particular we have that $U$ is contained in the union of $n$ lines, as we wanted to prove.  
 
 \subsection*{Acknowledgements}
 
 This work was supported by the full International Scholarship awarded by the Faculty of Graduate and Postdoctoral Studies (University of Ottawa, Canada).

\bibliography{biblio_rich_directions_v2}

\begin{thebibliography}{24}
\expandafter\ifx\csname natexlab\endcsname\relax\def\natexlab#1{#1}\fi
\providecommand{\url}[1]{\texttt{#1}}
\providecommand{\href}[2]{#2}
\providecommand{\path}[1]{#1}
\providecommand{\DOIprefix}{doi:}
\providecommand{\ArXivprefix}{arXiv:}
\providecommand{\URLprefix}{URL: }
\providecommand{\Pubmedprefix}{pmid:}
\providecommand{\doi}[1]{\href{http://dx.doi.org/#1}{\path{#1}}}
\providecommand{\Pubmed}[1]{\href{pmid:#1}{\path{#1}}}
\providecommand{\bibinfo}[2]{#2}
\ifx\xfnm\relax \def\xfnm[#1]{\unskip,\space#1}\fi
\bibitem[{R{\'e}dei(1970)}]{redei}
\bibinfo{author}{L.~R{\'e}dei}, \bibinfo{title}{L{\"u}ckenhafte Polynome
  {\"u}ber endlichen K{\"o}rpern, Akad{\'e}miai Kiad{\'o}},
  \bibinfo{publisher}{Budapest, and Birkh{\"a}user Verlag, Basel},
  \bibinfo{year}{1970}. \bibinfo{note}{(English translation: Lacunary
  polynomials over finite fields, Akad\'emiai Kiad\'o, Budapest, and North
  Holland, Amsterdam, 1973)}.
\bibitem[{Sz{\H{o}}nyi(1996)}]{szonyi}
\bibinfo{author}{T.~Sz{\H{o}}nyi},
\newblock \bibinfo{title}{{On the number of directions determined by a set of
  points in an affine Galois plane}},
\newblock \bibinfo{journal}{Journal of Combinatorial Theory, Series A}
  \bibinfo{volume}{74} (\bibinfo{year}{1996}) \bibinfo{pages}{141--146}.
\bibitem[{Sz{\H{o}}nyi(1999)}]{szonyi:around}
\bibinfo{author}{T.~Sz{\H{o}}nyi},
\newblock \bibinfo{title}{{Around R{\'e}dei's theorem}},
\newblock \bibinfo{journal}{{Discrete Mathematics}} \bibinfo{volume}{208-209}
  (\bibinfo{year}{1999}) \bibinfo{pages}{557--575}.
\bibitem[{Ghidelli and (https://mathoverflow.net/users/58242/luca ghidelli)(6
  06)}]{ghilu:MO}
\bibinfo{author}{L.~Ghidelli},
  \bibinfo{author}{(https://mathoverflow.net/users/58242/luca ghidelli)},
  \bibinfo{title}{{How many rich directions does a set in $\mathbb F_p^2$
  determine?}}, \bibinfo{howpublished}{MathOverflow},
  \bibinfo{year}{https://mathoverflow.net/q/299208 (version: 2018-06-06)}.
\bibitem[{Blokhuis et~al.(1995)Blokhuis, Brouwer, and
  Sz{\H{o}}nyi}]{directions:q}
\bibinfo{author}{A.~Blokhuis}, \bibinfo{author}{A.~E. Brouwer},
  \bibinfo{author}{T.~Sz{\H{o}}nyi},
\newblock \bibinfo{title}{The number of directions determined by a function f
  on a finite field},
\newblock \bibinfo{journal}{Journal of Combinatorial Theory, Series A}
  \bibinfo{volume}{70} (\bibinfo{year}{1995}) \bibinfo{pages}{349 -- 353}.
\bibitem[{Blokhuis et~al.(1999)Blokhuis, Ball, Brouwer, Storme, and
  Sz{\H{o}}nyi}]{directions:q:2}
\bibinfo{author}{A.~Blokhuis}, \bibinfo{author}{S.~Ball},
  \bibinfo{author}{A.~E. Brouwer}, \bibinfo{author}{L.~Storme},
  \bibinfo{author}{T.~Sz{\H{o}}nyi},
\newblock \bibinfo{title}{On the number of slopes of the graph of a function
  defined on a finite field},
\newblock \bibinfo{journal}{Journal of Combinatorial Theory, Series A}
  \bibinfo{volume}{86} (\bibinfo{year}{1999}) \bibinfo{pages}{187--196}.
\bibitem[{Fancsali et~al.(2013)Fancsali, Sziklai, and
  Tak{\'a}ts}]{directions:q:less}
\bibinfo{author}{S.~L. Fancsali}, \bibinfo{author}{P.~Sziklai},
  \bibinfo{author}{M.~Tak{\'a}ts},
\newblock \bibinfo{title}{The number of directions determined by less than q
  points},
\newblock \bibinfo{journal}{Journal of Algebraic Combinatorics}
  \bibinfo{volume}{37} (\bibinfo{year}{2013}) \bibinfo{pages}{27--37}.
\bibitem[{Sziklai and Tak{\'a}ts(2012)}]{directions:subspaces}
\bibinfo{author}{P.~Sziklai}, \bibinfo{author}{M.~Tak{\'a}ts},
\newblock \bibinfo{title}{An extension of the direction problem},
\newblock \bibinfo{journal}{Discrete Mathematics} \bibinfo{volume}{312}
  (\bibinfo{year}{2012}) \bibinfo{pages}{2083--2087}.
\bibitem[{De~Beule(2014)}]{directions:affine}
\bibinfo{author}{J.~De~Beule},
\newblock \bibinfo{title}{Direction problems in affine spaces},
\newblock in: \bibinfo{booktitle}{Galois geometries and applications},
  \bibinfo{organization}{Koninklijke Vlaams Academie van Belgi{\"e} voor
  Wetenschappen en Kunsten}, \bibinfo{year}{2014}, pp. \bibinfo{pages}{79--94}.
\bibitem[{Ball(2013)}]{polynomial:survey}
\bibinfo{author}{S.~Ball},
\newblock \bibinfo{title}{{The polynomial method in Galois geometries}},
\newblock \bibinfo{journal}{Current Research Topics in Galois Geometry}
  (\bibinfo{year}{2013}) \bibinfo{pages}{103}.
\bibitem[{Ball and Lavrauw(2004)}]{polynomial:higher}
\bibinfo{author}{S.~Ball}, \bibinfo{author}{M.~Lavrauw},
\newblock \bibinfo{title}{{How to use R{\'e}dei polynomials in higher
  dimensional spaces}},
\newblock \bibinfo{journal}{Le Matematiche} \bibinfo{volume}{59}
  (\bibinfo{year}{2004}) \bibinfo{pages}{39--52}.
\bibitem[{Blokhuis(1994)}]{polynomial:extremal}
\bibinfo{author}{A.~Blokhuis},
\newblock \bibinfo{title}{Extremal problems in finite geometry},
\newblock in: \bibinfo{booktitle}{Extremal problems for finite sets},
  \bibinfo{organization}{Janos Bolyai Mathematical Society},
  \bibinfo{year}{1994}, pp. \bibinfo{pages}{111--–135}.
\bibitem[{Blokhuis(1993)}]{polynomial:lecture}
\bibinfo{author}{A.~Blokhuis},
\newblock \bibinfo{title}{Polynomials in finite geometries and combinatorics},
\newblock in: \bibinfo{booktitle}{Surveys in combinatorics}, volume
  \bibinfo{volume}{187}, \bibinfo{publisher}{Cambridge Univ. Press Cambridge},
  \bibinfo{year}{1993}, pp. \bibinfo{pages}{35--52}.
\bibitem[{Sziklai(2013)}]{polynomial:thesis}
\bibinfo{author}{P.~Sziklai}, \bibinfo{title}{Applications of polynomials over
  finite fields}, Ph.D. thesis, ELTE TTK, \bibinfo{year}{2013}.
\bibitem[{Tao(2014)}]{otherpoly:tao}
\bibinfo{author}{T.~Tao},
\newblock \bibinfo{title}{Algebraic combinatorial geometry: the polynomial
  method in arithmetic combinatorics, incidence combinatorics, and number
  theory},
\newblock \bibinfo{journal}{EMS Surveys in Mathematical Sciences}
  \bibinfo{volume}{1} (\bibinfo{year}{2014}) \bibinfo{pages}{1--46}.
\bibitem[{Walsh(2018)}]{otherpoly:walsh}
\bibinfo{author}{M.~Walsh},
\newblock \bibinfo{title}{Characteristic subsets and the polynomial method},
\newblock in: \bibinfo{booktitle}{Proceedings of the International Congress of
  Mathematicians}, volume~\bibinfo{volume}{1}, \bibinfo{year}{2018}, pp.
  \bibinfo{pages}{465--484}.
\bibitem[{Dvir(2012)}]{otherpoly:dvir}
\bibinfo{author}{Z.~Dvir},
\newblock \bibinfo{title}{Incidence theorems and their applications},
\newblock \bibinfo{journal}{Foundations and Trends® in Theoretical Computer
  Science} \bibinfo{volume}{6} (\bibinfo{year}{2012})
  \bibinfo{pages}{257--393}.
\bibitem[{Bishnoi(2017)}]{otherpoly:thesis}
\bibinfo{author}{A.~Bishnoi}, \bibinfo{title}{Some contributions to incidence
  geometry and the polynomial method}, Ph.D. thesis, Ghent University, Belgium,
  \bibinfo{year}{2017}.
\bibitem[{Blokhuis et~al.(1999)Blokhuis, Storme, and
  Sz{\H{o}}nyi}]{blocking:multiple:baer}
\bibinfo{author}{A.~Blokhuis}, \bibinfo{author}{L.~Storme},
  \bibinfo{author}{T.~Sz{\H{o}}nyi},
\newblock \bibinfo{title}{Lacunary polynomials, multiple blocking sets and baer
  subplanes},
\newblock \bibinfo{journal}{Journal of the London Mathematical Society}
  \bibinfo{volume}{60} (\bibinfo{year}{1999}) \bibinfo{pages}{321--332}.
\bibitem[{Blokhuis et~al.(2007)Blokhuis, Lov{\'a}sz, Storme, and
  Sz{\H{o}}nyi}]{blocking:multiple}
\bibinfo{author}{A.~Blokhuis}, \bibinfo{author}{L.~Lov{\'a}sz},
  \bibinfo{author}{L.~Storme}, \bibinfo{author}{T.~Sz{\H{o}}nyi},
\newblock \bibinfo{title}{{On multiple blocking sets in Galois planes}},
\newblock \bibinfo{journal}{Advances in Geometry} \bibinfo{volume}{7}
  (\bibinfo{year}{2007}) \bibinfo{pages}{39--53}.
\bibitem[{Ball(1996)}]{blocking:multiple:arcs}
\bibinfo{author}{S.~Ball},
\newblock \bibinfo{title}{Multiple blocking sets and arcs in finite planes},
\newblock \bibinfo{journal}{Journal of the London Mathematical Society}
  \bibinfo{volume}{54} (\bibinfo{year}{1996}) \bibinfo{pages}{581--593}.
\bibitem[{Ferret et~al.(2012)Ferret, Storme, Sziklai, and
  Weiner}]{blocking:multiple:char}
\bibinfo{author}{S.~Ferret}, \bibinfo{author}{L.~Storme},
  \bibinfo{author}{P.~Sziklai}, \bibinfo{author}{Z.~Weiner},
\newblock \bibinfo{title}{A characterization of multiple (n-k)-blocking sets in
  projective spaces of square order},
\newblock \bibinfo{journal}{Advances in Geometry} \bibinfo{volume}{14}
  (\bibinfo{year}{2012}) \bibinfo{pages}{739--756}.
\bibitem[{Blokhuis et~al.(2011)Blokhuis, Sziklai, and Szonyi}]{blocking:nova}
\bibinfo{author}{A.~Blokhuis}, \bibinfo{author}{P.~Sziklai},
  \bibinfo{author}{T.~Szonyi},
\newblock \bibinfo{title}{Blocking sets in projective spaces},
\newblock \bibinfo{journal}{Current research topics in Galois geometry}
  (\bibinfo{year}{2011}) \bibinfo{pages}{61--84}.
\bibitem[{Lov{\'a}sz and Schrijver(1981)}]{redei:remarks}
\bibinfo{author}{L.~Lov{\'a}sz}, \bibinfo{author}{A.~Schrijver},
\newblock \bibinfo{title}{{Remarks on a theorem of R{\'e}dei}},
\newblock \bibinfo{journal}{Studia Scientiarum Mathematicarum Hungarica}
  \bibinfo{volume}{16} (\bibinfo{year}{1981}) \bibinfo{pages}{449--454}.

\end{thebibliography}

\end{document}